\documentclass[12pt]{amsproc}
\usepackage[T2A]{fontenc}      
\usepackage[utf8]{inputenc}   
\usepackage[russian,english]{babel} 
\usepackage{fullpage}
\usepackage{amsmath}
\usepackage{amsfonts}
\usepackage{amssymb}
\usepackage{amsthm}
\usepackage[utf8]{inputenc}
\usepackage{breqn}
\usepackage{afterpage}
\usepackage{longtable}
\usepackage{indentfirst}
\usepackage{caption}
\usepackage{subcaption}
\usepackage{comment} 
\usepackage{enumitem}

\usepackage{graphicx}
\graphicspath{ {./images/} }
\newtheorem{theorem}{Theorem}[section]
\newtheorem{lemma}[theorem]{Lemma}
\newtheorem{corollary}[theorem]{Corollary}
\newtheorem{proposition}[theorem]{Proposition}

\theoremstyle{definition}
\newtheorem{definition}[theorem]{Definition}
\newtheorem{example}[theorem]{Example}

\newtheorem{theorem-definition}[theorem]{Theorem-Definition}

\theoremstyle{remark}
\newtheorem{remark}[theorem]{Remark}

\numberwithin{equation}{section}

\usepackage{color}
\usepackage{xcolor}

\begin{document}

\title{ Weighted $\alpha$-subharmonic measure}

\begin{author}[Kuldoshev K. K.]{Kuldoshev K. K.}
    \address{National University of Uzbekistan,  Tashkent, Uzbekistan}
\email{qobiljonmath@gmail.com}
\end{author}
\begin{author}[Salaeva D. S.]{Salaeva D. S.}

\address{National University of Uzbekistan,  Tashkent, Uzbekistan}
\email{rafaelsanti476@gmail.com}
\end{author}
 
\begin{abstract}  

In this paper, we introduce and study the weighted $\alpha$-subharmonic measure associated with a weight function $\psi$, extending the usual $\alpha$-subharmonic measure and reducing to it when $\psi \equiv -1$.

Furthermore, we study the relationship between the weighted
$\alpha$-subharmonic measure and $\alpha$-regular compact sets. We also obtain a characterization of $(\alpha,\psi)$-regularity in
terms of the continuity of the corresponding weighted $\alpha$-subharmonic measure.
Finally,  we prove that if  the weighted $\alpha$-subharmonic measure of the compact set $K$ are Hölder continuous with respect to $K$, then it is Hölder continuous everywhere.  
\end{abstract}

\maketitle 

\section{Introduction}

The notion of an $\alpha$-subharmonic function, denoted by
$\alpha\text{-}sh$, is defined in terms of a closed strictly positive
differential form $\alpha$ of bidegree $(n-1,n-1)$ by the condition
$dd^c u \wedge \alpha \geq 0$
in the sense of currents (see Section~\ref{alfa subharmonic function} for the definition). This class of functions was originally introduced
as a useful tool in the study of  $m$-subharmonic functions
(see, for example, \cite{ABSA}, \cite{BZ}, \cite{DSKS1}, \cite{DSKS2}, \cite{AS},   for the definition of an
$m$-subharmonic function).
This notion provides greater flexibility and
technical convenience. It was later developed into an independent class of
functions and investigated by various authors. In particular, a number of
its  properties were established and systematically studied in
\cite{ABIS}, \cite{ABSR},  \cite{VM},  and related works.

Before stating our main result, we introduce some definitions. 
Let $D \subset \mathbb{C}^n$ be a bounded $\alpha$-regular domain 
(see Section~\ref{alpha subharmonic measure} for  the relevant definitions), 
let $K$ be a compact subset of $D$, and let $\psi$ be a bounded negative function defined on $K$ and admitting a
negative $\alpha$-subharmonic extension to $D$.

We denote by $\mathcal{U}(K,D,\psi)$ the class of all $\alpha$-subharmonic 
functions $u$ in $D$ satisfying
\[
u|_K \leq \psi|_K, \qquad u|_D < 0.
\]
We define
\[
\omega_\alpha(z,K,D,\psi)
=
\sup \{ u(z) : u \in \mathcal{U}(K,D,\psi) \}.
\]
The upper semicontinuous regularization
\[
\omega_\alpha^*(z,K,D,\psi)
=
\overline{\lim_{w\to z}}\, \omega_\alpha(w,K,D,\psi)
\]
is called the \textit{weighted $\alpha$-subharmonic measure} of the compact set $K$ 
with respect to $D$ and the weight $\psi$.
A point \(z_0\in K\) is called
\textit{\((\alpha,\psi)\)-regular with respect to \(D\)} if
\[
\omega_{\alpha}^*(z_0,K,D,\psi)=\psi(z_0).
\]
It is called \textit{locally \((\alpha,\psi)\)-regular with respect to \(D\)}
if, for every neighborhood \(B\subset \mathbb{C}^n\) of \(z_0\),
\[
\omega_{\alpha}^*(z_0,K\cap \overline{B},D,\psi)=\psi(z_0).
\]
The compact set \(K\) is called \textit{\((\alpha,\psi)\)-regular}
respectively, \textit{locally \((\alpha,\psi)\)-regular}, with respect to
\(D\), if each of its points has the corresponding property. Since \(D\) is fixed throughout the paper, for convenience, we simply say
\((\alpha,\psi)\)-regular, respectively locally \((\alpha,\psi)\)-regular,
compact sets, omitting the phrase “with respect to \(D\)”.

We now present the main results of this paper.  The first result characterizes $(\alpha,\psi)$-regularity via the continuity of
the weighted $\alpha$-subharmonic measure and gives its elliptic interpretation
as the solution of the corresponding boundary value problem.

\begin{theorem} \label{weighet continuity}
Let $D\subset \mathbb{C}^n$, $K\subset D$, and let $\psi$ and
$\omega_{\alpha}(z,K,D,\psi)$ be defined as above. Then the following statements hold:
\begin{enumerate}
\item The function $\omega_{\alpha}(z,K,D,\psi)$ is continuous in $D$ if and only if
$\psi \in C(K)$ and $K$ is an $(\alpha,\psi)$-regular compact set.

\item If $K$ is an $(\alpha,\psi)$-regular compact set, then
$\omega_{\alpha}(z,K,D,\psi)$ is the unique solution, in the class of
$\alpha$-subharmonic functions, of the following elliptic boundary value problem
\begin{equation}\label{HO}
\begin{cases}
dd^c u\wedge \alpha= 0, & \text{on } D \setminus K, \\[6pt]
\displaystyle\lim_{\substack{z \to \partial D \\ z \in D}} u(z) = 0, & \\[6pt]
u|_{K} = \psi. &
\end{cases}
\end{equation}
\end{enumerate}
\end{theorem}

The next theorem describes the relation between local $(\alpha,\psi)$-regularity
and local $\alpha$-regularity. It also shows that, under an additional extension
assumption on the weight function $\psi$, local $(\alpha,\psi)$-regularity is
equivalent to global $(\alpha,\psi)$-regularity.

\begin{theorem}\label{thm: regularity}
Let $K\subset D\subset \mathbb{C}^n$ be a compact set and let $\psi\in C(K)$.
Then the following statements hold:
\begin{enumerate}
    \item A point $z_0\in K$ is locally $(\alpha,\psi)$-regular if and only if
    it is locally $\alpha$-regular, i.e., for every neighborhood
    $B\subset D$ of $z_0$,
    \[
    \omega_{\alpha}^{*}(z_0,K\cap \overline{B},D)=-1.
    \]

    \item Assume that $\psi$ admits a negative strictly $\alpha$-subharmonic
    extension to a neighborhood $D^+$ of $\overline{D}$, i.e., there exists
    a function $\widetilde{\psi}$ which is strictly $\alpha$-subharmonic in
    $D^+$ and satisfies
    \[
    \widetilde{\psi}|_K=\psi|_K,
    \qquad
    \widetilde{\psi}<0 \quad \text{on } D.
    \]
\end{enumerate}
Then a point $z_0\in K$ is locally $(\alpha,\psi)$-regular if and only if it is
$(\alpha,\psi)$-regular.
\end{theorem}

It is well known that H\"older continuity plays an important role in the study
of dynamical systems of several complex variables. In \cite{KKRK}
(see, for example,  \cite{SAZA}, \cite{JS1} for complex Green functions in the class of
plurisubharmonic functions), it was proved that if the weighted
$m$-subharmonic measure $\omega_{m}^*(z,K,D,\psi)$ satisfies a H\"older-type
estimate with respect to the compact set $K$, then it is H\"older continuous
throughout the domain. In what follows, we present a theorem that extends this
result to the setting of weighted $\alpha$-subharmonic measures.
Let $D \subset \mathbb{C}^n$ be a strongly $\alpha$-regular domain
(see Section~\ref{alfa subharmonic function} for the definition).

\begin{theorem} \label{Holder}
Let $K\subset D$ be a compact set, and let $\psi:K\to\mathbb{R}$ be a
H\"older continuous function. Then $\omega_{\alpha}(z,K,D,\psi)$ is
H\"older continuous in $D$ if and only if there exist constants $C>0$ and
$0<\lambda\leq 1$ such that
\begin{equation}\label{Holdercondition}
\left|\omega_{\alpha}(z,K,D,\psi)-\psi(w)\right|
\leq C(\operatorname{dist}(z,K))^{\lambda}
\end{equation}
holds for every $z$ in some neighborhood of $K$, where  $ w\in K $ is a point closest to $z$, i.e., $|z-w|=\mathrm{dist}(z,K).$
\end{theorem}

It should be emphasized that the three theorems stated above, which constitute the main results of this paper, are new even in the unweighted case, i.e., when $\psi \equiv -1$.

This paper is organized as follows.
In Section~\ref{alfa subharmonic function}, we recall the definition of $\alpha$-subharmonic
functions and collect several preliminary notions and properties. We also
establish a Hartogs-type lemma for families of $\alpha$-subharmonic functions.
In Section~\ref{alpha subharmonic measure}, we study the $\alpha$-subharmonic measure and state several new functional properties.
In Section~\ref{weighted alpha measure}, we introduce the weighted $\alpha$-subharmonic measure
associated with a weight function $\psi$ and establish its main properties,
including monotonicity, convergence properties, and $\alpha$-harmonicity
outside the given set.
In Section~\ref{regularity of compacts}, we investigate $(\alpha,\psi)$-regular compact sets.
In particular, we prove Theorem~\ref{weighet continuity} and Theorem~\ref{thm: regularity}.
Finally, in Section~\ref{Holder section}, we prove Theorem~\ref{Holder}, which gives
a necessary and sufficient condition for the H\"older continuity of the function $\omega_{\alpha}(z,K,D,\psi)$ in the whole domain.

\textbf{Acknowledgments.}
The authors would like to thank the V.I. Romanovskiy Institute of Mathematics
of the Academy of Sciences of the Republic of Uzbekistan for its warm welcome
and excellent working conditions.
This work has received funding from the Ministry of Higher Education, Science and Innovation of the Republic of Uzbekistan under State Grant No. FL-9524115114.

\section{$\alpha$-subharmonic functions and Hartogs lemma for a family of $\alpha$-subharmonic function} \label{alfa subharmonic function}

 Let $D$ be a bounded domain in $\mathbb{C}^n$ and $\alpha $ be closed and strictly positive differential form of degree $(n-1,n-1)$ in $D$
\begin{equation}\label{alfa dif for}
\alpha = \left(\frac{i}{2}\right)^{n-1} \sum_{j,k=1}^{n} {\alpha_{jk}(z) \ dz[j] \wedge d\overline{z}[k]}, 
\end{equation}
where 
$\alpha_{jk} \in C^1(D)$, 
$$dz[j]=dz_{1}\wedge dz_{2}...\wedge dz_{j-1}\wedge dz_{j+1}\wedge...\wedge dz_n, $$
$$d\overline{z}[k]=d\overline{z}_{1}\wedge d\overline{z}_{2}...\wedge d\overline{z}_{k-1}\wedge d\overline{z}_{k+1}\wedge...\wedge d\overline{z}_{n}.$$

    \begin{definition}[see \cite{ABIS}] \label{alpha-definition}
A function $u \in L_{loc}^{1} (D)$ is called \textit{ $\alpha$-subharmonic} in $D$, if it is strongly upper semicontinuous in $D$, i.e., $u(z)$ is upper semicontinuous in $D$, in addition, for every $A \subset D$ which has full Lebesgue measure in a neighbourhood of $z^0 \in D$, the following relation

\[
\limsup_{\substack{z\to z^0\\ z\in A}} u(z)=u(z^0)
\]
holds and 
 the current  $dd^{c}u \wedge \alpha$ is positive i.e.,
for any positive test function $\omega$ in $D$, we have 
\begin{equation} \label{current}
    \int u \alpha\wedge dd^c\omega \ge 0, 
\end{equation}
where $d=\partial +\overline{\partial}, \, d^c=\frac{\overline{\partial}-\partial}{4i}$ (for further details on currents, see \cite{JP}).
Moreover,
 a function $u(z)$ is called \textit{strictly $\alpha$-subharmonic} in $D$, if for any compact domain $G \Subset D$, there exists $\varepsilon>0$ such that the function $ u(z)- \varepsilon|z|^2$ is $\alpha$-subharmonic in $G$.
\end{definition}

Note that if \(u \in C^2(D)\), then the condition \(dd^c u \wedge \alpha \geq 0\) is equivalent to the above condition \eqref{current}.
The class of $\alpha$-subharmonic functions in $D$ is denoted by
$\alpha\text{-}sh(D)$ and for convenience, the function
$u(z)\equiv -\infty$ is also considered to belong to this class.

When  $\alpha=\beta^{n-1}$, where $\beta=dd^c|z|^2$, the class of $\alpha$-subharmonic functions coincides with the class of subharmonic functions. In addition, every plurisubharmonic function is $\alpha$-subharmonic, i.e., $psh(D) \subset \alpha\text{-}sh(D)$. However, the subharmonic and $\alpha$-subharmonic functions are different classes, even in $\mathbb{C}^2$. For example, let us be given a differential form $\alpha=\frac{i}{2}(3 dz_1 \wedge d \overline{z}_1 +dz_2 \wedge d \overline{z}_2)$. Consider the following functions:

\begin{itemize} 
    \item $u_1(z)=2|z_1|^2-|z_2|^2$ is subharmonic, but not $\alpha$-subharmonic;
    \item $u_2(z)=-3|z_1|^2+2|z_2|^2$ is $\alpha$-subharmonic, but not subharmonic;
    \item $u_3(z)=|z_1|^2+|z_2|^2$ is $\alpha$-subharmonic and subharmonic.
\end{itemize}

  The class of $\alpha$-subharmonic functions has been proven to have the following properties(see \cite{ABIS}, \cite{ABSR}):

\begin{enumerate}
  \item  convex linear combinations of $\alpha$-subharmonic functions is $\alpha$-subharmonic, i.e., 
  \[
 \lambda, \, \mu \ge0, \,  u_1, u_2 \in \alpha\text{-}sh(D) \implies \lambda u_1+\mu u_2 \in \alpha\text{-}sh(D);
  \]

  \item  the limit function of the decreasing or uniformly convergence sequence of $\alpha$-subharmonic functions is $\alpha$-subharmonic, i.e.,
  \[
  \{u_j\} \subset \alpha\text{-}sh(D), \, \, u_j(z) \downarrow u(z)  \implies u \in \alpha\text{-}sh(D);
  \]
\[
\{u_j\} \subset \alpha\text{-}sh(D), \, \, u_j \rightrightarrows u \implies u \in \alpha\text{-}sh(D);
\]

\item maximum of the finite number of $\alpha$-subharmonic functions is $\alpha$-subharmonic, i.e.,
\[
u_1, u_2 \in \alpha\text{-}sh(D), u(z)=\max\{u_1(z),u_2(z)\} \implies u \in \alpha\text{-}sh(D);
\]
     
 \item \label{family} let $\{u_{\lambda}\}, \, \lambda \in \Lambda$ be a family of locally uniformly bounded $\alpha$-subharmonic functions and $u(z)=\sup\limits_{\lambda \in \Lambda}u_{\lambda}(z)$. Then, the upper regularization $u^*(z)$ of the function $u(z)$ is $\alpha$-subharmonic in $D$.

\item \label{maximum principle}
An $\alpha$-subharmonic function cannot attain its maximum in the domain $D$
unless it is constant in $D$.
 \end{enumerate}

 In \cite{ABIS}, it has been shown that the differential operator $\Delta_{\alpha}u=dd^cu \wedge \alpha$ defined in the class $C^2(D)$ is self-adjoint (see \cite{MR}, for definition of self-adjoint). Moreover, it is not difficult to show that, for a strictly positive differential form $\alpha$, the operator $dd^cu \wedge \alpha$ is elliptic. It is known from the theory of elliptic operators that the operator $\Delta_{\alpha}$ has a fundamental solution $K(z,w)$ and it represents many properties (see \cite{DT}, \cite{MR}) which can be used to learn the class of $\alpha$-subharmonic functions. For example, using the fundamental solution $K(z,w)$, any function $u \in \alpha\text{-}sh(D)$ can be written in the following form in $G \Subset D$
\[
u(z)=\int\limits_G K(z,w) d \mu(w) +g(z),
\]
where $\mu=dd^cu \wedge \alpha$ and the function $g(z)$ is $\alpha\text{-}$harmonic, i.e., $dd^cg \wedge \alpha=0$ in $G$ (see \cite{ABIS}, \cite{VM}).

We recall the following facts from \cite{ABIS}. One of the important properties
of $\alpha$-subharmonic functions is an equivalent characterization in terms
of the Poisson kernel $P_{\alpha}(z,\xi)$ of a ball. The kernel $P_{\alpha}(z,\xi)$, associated with the above fixed differential form $\alpha$, has properties analogous to those of the classical Poisson kernel of a ball. More precisely, for every fixed
$\xi\in \partial B(z^0,r)$, the function $P_{\alpha}(z,\xi)$ is
$\alpha$-harmonic with respect to $z\in B(z^0,r)$, and for every fixed
$z\in B(z^0,r)$, it is continuous with respect to
$\xi\in \partial B(z^0,r)$. Moreover,
\[
P_{\alpha}(z,\xi)\geq 0.
\]

If $u$ is $\alpha$-subharmonic in $D$, then for every 
ball $B(z^0,r)\Subset D$, the following Poisson-type inequality holds
\begin{equation}\label{puasson}
u(z^0)\leq
\int_{\partial B(z^0,r)} u(\xi)P_{\alpha}(z^0,\xi)\,d\sigma(\xi),  \quad \forall z \in B(z^0,r),
\end{equation}
where $P_{\alpha}(z,\xi)$ denotes the Poisson kernel of the ball
$B(z^0,r)$. Conversely, if an upper semicontinuous function $u$ satisfies
\eqref{puasson} for every ball $B(z^0,r)\Subset D$, then $u$ is
$\alpha$-subharmonic in $D$.
Thus, condition \eqref{puasson} gives a local characterization of
$\alpha$-subharmonic functions, since it depends only on the behaviour of the
function in a neighbourhood of the given point. In particular, this
characterization can be used to prove local properties of
$\alpha$-subharmonic functions. Additionally, in \cite{ABIS}, it has been shown that the condition \eqref{puasson} is equivalent to 
\[
    u(z) \leq \int\limits_{\partial B(z^0,r)}P_{\alpha}(z,\xi)u(\xi) d \sigma(\xi),  \quad \forall z \in B(z^0,r)
 \]   

Using equivalent definition of $\alpha$-subharmonic functions, now we show that the Hartogs' lemma, which plays an important role in potential theory, also holds for $\alpha$-subharmonic functions.

\begin{lemma} \label{Hartogs}
Let $\{u_j\}$ be a sequence of $\alpha$-subharmonic functions in a domain $D \subset \mathbb{C}^n$, which is locally uniformly bounded from above. Assume that for every $z \in D$,
\[
\limsup_{j\to\infty} u_j(z)\le g(z),
\]
where $g\in C(D)$. Then for every compact set $K\Subset D$ and every $\varepsilon>0$, there exists $j_0\in\mathbb N$ such that
\[
u_j(z)\le g(z)+\varepsilon
\qquad \text{for all } z\in K \text{ and all } j\ge j_0.
\]
\end{lemma}

\begin{proof}
Fix a compact set $K\Subset D$ and let $z^0\in K$. Since $g$ is continuous on $D$, there exists $\delta>0$ such that
$\overline{B(z^0,\delta)}\Subset D$
and
\[
|g(z)-g(z^0)|<\frac{\varepsilon}{4},
\qquad \forall z\in \overline{B(z^0,\delta)}.
\]
Since the sequence $\{u_j\}$ is locally uniformly bounded from above, there exists a constant $M>0$ such that
\[
u_j(z)\le M
\qquad \text{for all } z\in \overline {B(z^0,\delta)}, \quad j\in\mathbb N.
\]
By the $\alpha$-subharmonicity of $u_j$ and the inequality \eqref{puasson}, we have
\[
u_j(z^0)\le \int_{\partial B(z^0,\delta)} P_\alpha(z^0,\xi)\,u_j(\xi)\,d\sigma(\xi),
\]
where $P_{\alpha}(z,\xi)$ is the Poisson kernel of the ball $B(z^0,\delta)$. Passing to the limit superior and using Fatou's lemma for the nonnegative functions
\[
P_\alpha(z^0,\xi)\bigl(M-u_j(\xi)\bigr),
\]
we obtain
\[
\limsup_{j\to\infty}\int_{\partial B(z^0,\delta)} P_\alpha(z^0,\xi)\,u_j(\xi)\,d\sigma(\xi)
\le
\int_{\partial B(z^0,\delta)} P_\alpha(z^0,\xi)\,\limsup_{j\to\infty}u_j(\xi)\,d\sigma(\xi).
\]
Hence,
\[
\limsup_{j\to\infty}u_j(z^0)
\le
\int_{\partial B(z^0,\delta)} P_\alpha(z^0,\xi)\,\limsup_{j\to\infty}u_j(\xi)\,d\sigma(\xi)
\le
\int_{\partial B(z^0,\delta)} P_\alpha(z^0,\xi)\,g(\xi)\,d\sigma(\xi).
\]
Since $|g(\xi)-g(z^0)|<\varepsilon/4$ on $\partial B(z^0,\delta)$ and
\[
\int_{\partial B(z^0,\delta)} P_\alpha(z^0,\xi)\,d\sigma(\xi)=1,
\]
it follows that
\[
\int_{\partial B(z^0,\delta)} P_\alpha(z^0,\xi)\,g(\xi)\,d\sigma(\xi)
\le g(z^0)+\frac{\varepsilon}{4}.
\]
Consequently, there exists $j_0(z^0)\in\mathbb N$ such that for all $j\ge j_0(z^0)$,
\[
\int_{\partial B(z^0,\delta)} P_\alpha(z^0,\xi)\,u_j(\xi)\,d\sigma(\xi)
\le g(z^0)+\frac{\varepsilon}{2}.
\]
Since $P_\alpha(z,\xi)$ is continuous on
\[
\overline{B(z^0,\delta/2)}\times \partial B(z^0,\delta),
\]
it is uniformly continuous there. Hence there exists $\delta_1\in (0,\delta/2)$ such that for all $z\in B(z^0,\delta_1)$ and all $\xi\in \partial B(z^0,\delta)$,
\[
|P_\alpha(z,\xi)-P_\alpha(z^0,\xi)|
<
\frac{\varepsilon}{4M\,\sigma_{2n}\,\delta^{2n-1}},
\]
where $\sigma_{2n}=\sigma(\partial B(0,1))$ denotes the surface measure of the unit sphere in $\mathbb C^n\simeq\mathbb R^{2n}$.
Hence, for all $z\in B(z^0,\delta_1)$ and all $j\ge j_0(z^0)$,
\[
u_j(z)
\le
\int_{\partial B(z^0,\delta)} P_\alpha(z,\xi)\,u_j(\xi)\,d\sigma(\xi)
\]
\[
\le
g(z^0)+\frac{\varepsilon}{2}
+
\frac{\varepsilon}{4M\,\sigma_{2n}\,\delta^{2n-1}}
\int_{\partial B(z^0,\delta)} u_j(\xi)\,d\sigma(\xi) \le
g(z^0)+\frac{3\varepsilon}{4}.
\]
Since $|g(z)-g(z^0)|<\varepsilon/4$ for $z\in B(z^0,\delta_1)$, it follows that
\[
u_j(z)\le g(z)+\varepsilon
\qquad \text{for all } z\in B(z^0,\delta_1), \ j\ge j_0(z^0).
\]
Now the family of balls $\{B(z,\delta_1(z))\}_{z\in K}$ forms an open cover of the compact set $K$. Hence there exist finitely many points
\[
z^{(1)},\dots,z^{(m)}\in K
\]
such that
\[
K\subset \bigcup_{\ell=1}^m B\bigl(z^{(\ell)},\delta_1(z^{(\ell)})\bigr).
\]
Then for every $j\ge j_0, j_0=\max_{1\le \ell\le m} j_0\bigl(z^{(\ell)}\bigr) $ and every $z\in K$, we have
\[
u_j(z)\le g(z)+\varepsilon.
\]
The proof is complete.
\end{proof}
 \begin{remark}
 The method used in the proof of Hartogs' lemma for the class of subharmonic functions does not apply to the proof of Hartogs' lemma in the class of $\alpha$-subharmonic functions. However, the proof given above provides a new proof of Hartogs' lemma even in the class of subharmonic functions.
 \end{remark}

\section{$\alpha$-subharmonic measure} \label{alpha subharmonic measure}

The $\alpha$-subharmonic measure is defined as an extremal function in the class of $\alpha$-subharmonic functions.
The $\alpha$-subharmonic measure is usually defined in $\alpha$-regular
domains, which may be regarded as a generalization of strongly pseudoconvex
domains.

\begin{definition}\label{SMRD}
A domain $D \subset \mathbb{C}^n$ is called \textit{$\alpha$-regular} if there exists a function $\rho \in \alpha\text{-}sh({D})$ such that $\rho|_D < 0,\ \lim\limits_{z \to \partial D} \rho(z) = 0$.
It is called \textit{strongly $\alpha$-regular} if  $\rho \in \alpha \text{-}sh(D^+) \cap C^2(D^+)$ and is strictly $\alpha$-subharmonic in $D^+$, where $D^+$ is a neighborhood of the closure $\bar{D}$.
\end{definition}

Let $D \subset \mathbb{C}^n$ be a bounded $\alpha$-regular domain and
 $E \subset D$ be any subset, and denote by $\mathcal{U}(E,D)$ the class of functions $u \in \alpha\text{-}sh(D)$ satisfying
\[
u|_E \leq -1, \qquad u < 0 \quad \text{on } D.
\]
Define the function
\[
\omega_{\alpha}(z,E,D) := \sup \{ u(z) : u \in \mathcal{U}(E,D) \}.
\]
It is well known that a function defined in this way is not necessarily upper semicontinuous. To ensure its upper semicontinuity, we consider its upper semicontinuous regularization.
\begin{definition}\label{def of asm}
The upper semicontinuous regularization of the function $\omega_{\alpha}(z,E,D)$, defined by
\[
\omega^*_{\alpha}(z,E,D) := \limsup_{w \to z} \, \omega_{\alpha}(w,E,D),
\]
is called \textit{the $\alpha$-subharmonic measure} (or $P_{\alpha}$-measure) of the set $E$ with respect to the domain $D$.
\end{definition}

By the  property \ref{family} of $\alpha$-subharmonic functions in Section~\ref{alfa subharmonic function}, it is obvious that the $\alpha$-subharmonic measure is $\alpha$-subharmonic in $D$ and  satisfies the inequality: 
$$-1 \leq \omega_{\alpha}^*(z,E,D) \leq 0, \quad \forall z \in D.$$ 
By the maximum principle for $\alpha$-subharmonic functions, the $\alpha$-subharmonic measure is either nowhere or everywhere zero.
In addition, if a function $u(z)$ is $\alpha$-subharmonic in $D$ and satisfies
\[
u|_E \leq r, \qquad u < R \quad \text{on } D,
\]
where $R > r$, then the function
\[
\frac{u(z)-R}{R-r}
\]
belongs to the class $\mathcal{U}(E,D)$. Hence,
\[
\frac{u(z)-R}{R-r} \leq \omega_{\alpha}^*(z,E,D),
\]
which is equivalent to
\begin{equation} \label{two constants}
u(z) \leq R\bigl(1+\omega_{\alpha}^*(z,E,D)\bigr) - r\,\omega_{\alpha}^*(z,E,D),
\end{equation}
for all $z \in D$. The inequality \eqref{two constants} is called the theorem on two constants.

In addition, the $P_{\alpha}$-measure satisfies the following simple
properties (see \cite{SR}).

\begin{proposition}\label{prop: properties of alpha-measure}
The following statements hold:
\begin{enumerate}
\item If $E_1 \subset E_2 \subset D$, then
\[
\omega_{\alpha}^*(z,E_1,D) \ge \omega_{\alpha}^*(z,E_2,D), \quad  \text{for all } z \in D.
\]
If $E \subset D_1 \subset D_2$, then
\[
\omega_{\alpha}^*(z,E,D_1) \ge \omega_{\alpha}^*(z,E,D_2) \quad \text{for all } z \in D_1.
\]

\item If $U$ is an open subset of $D$, then
\[
\omega_{\alpha}^*(z,U,D) = \omega_{\alpha}(z,U,D)
\quad \text{for all } z \in D.
\]

\item Let $U$ be an open set such that $U = \bigcup_{j=1}^{\infty} K_j$, where $K_j \subset K_{j+1}$ and each $K_j$ is compact. Then
\[
\omega_{\alpha}^*(z,K_j,D) \downarrow \omega_{\alpha}(z,U,D).
\]

\item If $E \subset D$, then there exists a decreasing sequence of open sets $\{U_j\}$ with $E \subset U_j$ such that
\[
\omega_{\alpha}^*(z,E,D) = \left(\lim_{j \to \infty} \omega_{\alpha}(z,U_j,D)\right)^*.
\]

\item Let $E = \bigcup_{j=1}^{\infty} E_j$, where $E_j \subset D$. Then, for all $z \in D$,
\[
\omega_{\alpha}^*(z,E,D) \ge \sum_{j=1}^{\infty} \omega_{\alpha}^*(z,E_j,D).
\]

\item The $\alpha$-subharmonic measure $\omega_{\alpha}^*(z,E,D)$ is equal to zero if and only if $E$ is $\alpha$-polar in $D$, that is, there exists
a function $\rho\in \alpha\text{-}sh(D)$ such that
$\rho|_E=-\infty$ and $\rho\not\equiv -\infty$.
\end{enumerate}
\end{proposition}

Nevertheless, there remain numerous properties of the $\alpha$-subharmonic measure that have not yet been thoroughly investigated. The study of these properties continues to be of considerable interest in modern analysis. 
Below, we present a proposition describing several important properties of the $\alpha$-subharmonic measure related to sets that have not yet been studied.
\begin{proposition} \label{Pro sets}
The following properties hold for the $\alpha$-subharmonic measure:

\begin{enumerate}

\item Let $A\subset D$ be an $\alpha$-polar set in $D$. Then, for every
$E\subset D$, we have
\[
\omega_{\alpha}^*(z,E,D)=\omega_{\alpha}^*(z,E\cup A,D),
\]
for all $z\in D$.

\item If $E \subset \subset D$, then
\[
\lim_{z \to \partial D} \omega_{\alpha}^*(z,E,D) = 0.
\]

\item Let $\{K_j\}$ be a sequence of compact sets such that $K_{j+1} \subset K_j$ and $K = \bigcap_{j=1}^{\infty} K_j$. Then, for all $z \in D$,
\[
\lim_{j \to \infty} \omega_{\alpha}(z,K_j,D) = \omega_{\alpha}(z,K,D).
\]

\item Let $E \subset \subset D_1$ and let $\{D_j\}$ be an increasing sequence of domains such that \\ $D_j \subset D_{j+1}$ and $D = \bigcup_{j=1}^{\infty} D_j$. Then,
\[
\lim_{j \to \infty} \omega_{\alpha}^*(z,E,D_j) = \omega_{\alpha}^*(z,E,D).
\]

\end{enumerate}
\end{proposition}

The proof of Proposition~\ref{Pro sets} will be given in Section 4 for the weighted $\alpha$-subharmonic measure, which is a generalization of the $\alpha$-subharmonic measure (see Proposition \ref{extends} in Section 4).

Assume that $K \subset D$ is a compact set.

\begin{definition} \label{reg}
A point $z_0\in K$ is called an \textit{$\alpha$-regular point of $K$ with respect to $D$} if
\[
\omega_{\alpha}^*(z_0,K,D)=-1
\]
It is called a \textit{locally $\alpha$-regular point of $K$ with respect to $D$} if, for every
neighborhood $B\subset \mathbb{C}^n$ of $z_0$, one has
\[
\omega_{\alpha}^*(z_0,K\cap \overline{B},D)=-1.
\]
If every point of $K$ is $\alpha$-regular, then $K$ is called an
\textit{$\alpha$-regular compact set with respect to $D$}. If every point of $K$ is
locally $\alpha$-regular, then $K$ is called a
\textit{locally $\alpha$-regular compact set with respect to $D$}.
\end{definition}

Throughout this paper, the domain $D$ is assumed to be fixed. Hence, for simplicity,
we refer to compact sets that are $\alpha$-regular with respect to $D$
(respectively, locally $\alpha$-regular with respect to $D$) simply as
$\alpha$-regular (respectively, locally $\alpha$-regular) compact sets.

By Definition~\ref{reg}, every locally $\alpha$-regular point of a compact set
$K$ is $\alpha$-regular. Consequently, every locally $\alpha$-regular compact
set is $\alpha$-regular.

However, the converse is not true in general, as the following simple example shows.

\begin{example}
Let $D=B(0,2)$ be a ball in $\mathbb{C}^2$, let
$K=\{z\in \mathbb{C}^{2}: |z|=1\}\cup\{0\}$, and let
\(\alpha=dd^{c}|z|^{2}\). It is well known that \(K\) is \(\alpha\)-regular
(see, for example, \cite{AS}), i.e.,
\[
\omega_{\alpha}^*(z,K,D)= -1,
\quad \forall z\in K.
\]
On the other hand, the point \(0\) is not locally \(\alpha\)-regular. Indeed,
for every ball \(B=B(0,r)\) with \(0<r<1\), we have
$K\cap \overline{B}=\{0\}$,
and hence
\[
\omega_{\alpha}^*(z,K\cap \overline{B},D)
\equiv 0
\qquad \text{in } D.
\]
\end{example}

We present the following lemma describing the relationship between $\alpha$-regularity and pluriregularity.

\begin{lemma}\label{pluriregular}
Every $\alpha$-regular (respectively, locally $\alpha$-regular) compact set is pluriregular (respectively, locally pluriregular).
\end{lemma}

\begin{proof}
Let $K \subset D$ be an $\alpha$-regular compact set. Since every plurisubharmonic function is $\alpha$-subharmonic, it follows that the plurisubharmonic measure of $K$ does not exceed its $\alpha$-subharmonic measure, i.e.,
\[
\omega^*(z, K, D) \le \omega_{\alpha}^*(z, K, D), \qquad \forall z \in D,
\]
where $\omega^*(z, K, D)$ denotes the plurisubharmonic measure of the compact set $K$ with respect to the domain $D$ (see, for example, \cite{AS}, \cite{SAA}).
Since $K$ is $\alpha$-regular, we have 
\[\omega_{\alpha}^*(z^0, K, D) = -1\]
for every $z^0 \in K$. Hence,
\[
\omega^*(z^0, K, D) = -1, \qquad \forall z^0 \in K.
\]
Therefore, $K$ is pluriregular.
The proof of the statement concerning local regularity follows by the same argument.
The proof is complete.
\end{proof}

We have the following proposition describing the relationship between
$\alpha$-regularity and continuity of the $\alpha$-subharmonic measure.

\begin{proposition}\label{continuity}
The function $\omega_{\alpha}(z,K,D)$ is continuous in $D$ if and only if
$K$ is an $\alpha$-regular compact set.
\end{proposition}

We will establish this proposition in a more general form for the weighted
$\alpha$-subharmonic measure in Section~5, from which the above result follows
as an immediate consequence.

 \section{The weighted $\alpha\text{-}$subharmonic measure}\label{weighted alpha measure}
 In this section, we introduce the notion of the weighted $\alpha$-subharmonic measure, a generalization of the $\alpha$-subharmonic measure, and establish several of its important properties.
 
 Let us first recall some necessary definitions.
Let $D \subset \mathbb{C}^n$ be a bounded $\alpha$-regular domain, let $E \subset D$ be a fixed subset, and let $\psi$ be a bounded function on $E$ such that
\[
\sup_{z \in E} \psi(z) < 0.
\]
We denote by $\mathcal{U}(E,D,\psi)$ the class of all functions $u \in \alpha\text{-}sh(D)$ satisfying
\[
u|_E \leq \psi_E, \qquad u < 0 \quad \text{on } D.
\]
We define the extremal function
\[
\omega_\alpha(z,E,D,\psi)
=
\sup \{ u(z) : u \in \mathcal{U}(E,D,\psi) \}.
\]

\begin{definition} \label{weighted measure}
The upper semicontinuous regularization of $\omega_\alpha(z,E,D,\psi)$, defined by
\[
\omega_\alpha^*(z,E,D,\psi)
=
\limsup_{w \to z} \omega_\alpha(w,E,D,\psi),
\]
is called \textit{the $(\alpha,\psi)$-subharmonic measure} of the set $E$ with respect to the domain $D$.
\end{definition}

By Definition~\ref{weighted measure} and property~\ref{family} of $\alpha$-subharmonic functions in the Section~\ref{alfa subharmonic function}, the function $\omega_\alpha^*(z,E,D,\psi)$ is $\alpha$-subharmonic in $D$ and satisfies
\[
\inf_E \psi(w) \leq \omega_\alpha^*(z,E,D,\psi) \leq 0,
\qquad \forall z \in D.
\]
Note that if $\psi \equiv -1$, then the $(\alpha,\psi)$-subharmonic measure coincides with the $\alpha$-subharmonic measure, i.e.,
\[
\omega_\alpha^*(z,E,D,-1) \equiv \omega_\alpha^*(z,E,D).
\]

Moreover, by the maximum principle for $\alpha$-subharmonic functions, the $\alpha$-subharmonic measure is either identically zero or nowhere zero in $D$.

Furthermore, if
\[
\inf_{z \in E} \psi(z) \geq 0,
\]
then, by the maximum principle, the role of the weight function $\psi$ in the definition of the weighted $\alpha$-subharmonic measure disappears, and
\[
\omega_\alpha^*(z,E,D,\psi) \equiv 0
\]
in $D$. Therefore, throughout the paper, when defining
$\omega_\alpha^*(z,E,D,\psi)$, we assume that $\psi$ is bounded on $E \subset D$ and satisfies
\[
\sup_{z \in E} \psi(z) < 0.
\]

The following lemmas are immediate consequences of the definition of the $(\alpha,\psi)$-subharmonic measure.

\begin{lemma} \label{prop: monotonicity}
The following monotonicity properties hold:
\begin{enumerate}
\item If $E_1 \subset E_2 \subset D_1 \subset D_2$, and $\psi$ is defined on $E_2$, then
\[
\omega_\alpha^*(z,E_2,D_2,\psi)
\leq
\omega_\alpha^*(z,E_1,D_2,\psi)
\leq
\omega_\alpha^*(z,E_1,D_1,\psi),
\qquad \forall z \in D_1.
\]

\item If $\psi_1|_E \leq \psi_2|_E$, then
\[
\omega_\alpha^*(z,E,D,\psi_1)
\leq
\omega_\alpha^*(z,E,D,\psi_2),
\qquad \forall z \in D.
\]
\end{enumerate}
\end{lemma}

\begin{lemma}[Theorem on two constants]\label{prop: two constant theorem for weighted}
Let $u \in \alpha\text{-}sh(D)$. Suppose that
\[
u < M \quad \text{on } D, 
\qquad 
u|_E \leq m,
\]
where $m < M$ and $E \subset D$. Then
\[
u(z) \leq
M \left(1-\frac{\omega_\alpha^*(z,E,D,\psi)}{\inf\limits_{w \in E}\psi(w)}\right)
+
m \frac{\omega_\alpha^*(z,E,D,\psi)}{\inf\limits_{w \in E}\psi(w)},
\qquad \forall z \in D.
\]
\end{lemma}

The proof of Lemma~\ref{prop: two constant theorem for weighted} follows from the fact that
\[
\frac{u-M}{m-M}\,\inf_{w \in E}\psi(w)
\in \mathcal{U}(E,D,\psi).
\]

Now we present a relation that expresses the connection between the weighted and unweighted $\alpha$-subharmonic measures.

\begin{lemma} \label{connection inequality}
For any set $E \subset D$, the following inequalities hold for all $z \in D$
\[
-\inf_{w \in E}\psi(w)\,\omega^*_{\alpha}(z,E,D)
\leq
\omega_\alpha^*(z,E,D,\psi)
\leq
-\sup_{w \in E}\psi(w)\,\omega^*_{\alpha}(z,E,D).
\]
\end{lemma}

\begin{proof}
Take an arbitrary function $u \in \mathcal{U}(E,D)$, that is,
\[
u \in \alpha\text{-}sh(D), \qquad u<0 \ \text{on } D, \qquad u|_E \leq -1.
\]
Since $\sup\limits_{w \in E}\psi(w)<0$, we have 
\[
\left(-\inf_{w \in E}\psi(w)\,u(z)\right)\Big|_D<0
\]
and
\[
\left(-\inf_{w \in E}\psi(w)\,u(z)\right)\Big|_E
\leq
\inf_{w \in E}\psi(w)
\leq
\psi|_E.
\]
Hence,
\[
-\inf_{w \in E}\psi(w)\,u\in \mathcal{U}(E,D,\psi)
\]
and then
\[
-\inf_{w \in E}\psi(w)\,u(z)
\leq
\omega_\alpha^*(z,E,D,\psi),
\qquad \forall z \in D.
\]
Since $u\in \mathcal{U}(E,D)$ is arbitrary, we obtain
\[
-\inf_{w \in E}\psi(w)\,\omega_\alpha^*(z,E,D)
\leq
\omega_\alpha^*(z,E,D,\psi),
\qquad \forall z \in D.
\]

It remains to prove the reverse estimate
\[
\omega_\alpha^*(z,E,D,\psi)
\leq
-\sup_{w \in E}\psi(w)\,\omega_\alpha^*(z,E,D).
\]
Take an arbitrary function $u \in \mathcal{U}(E,D,\psi)$ and consider the  function
$-\frac{u(z)}{\sup\limits_{w \in E}\psi(w)}.$
Since $\sup\limits_{w \in E}\psi(w)<0$, this function is $\alpha$-subharmonic in $D$. Moreover,
\[
-\frac{u(z)}{\sup\limits_{w \in E}\psi(w)}<0
\quad \forall z \in  D,
\quad \text{and} \quad
-\frac{u(z)}{\sup\limits_{w \in E}\psi(w)}
\leq -1, \quad \forall z \in E.
\]
Thus, the function $-\frac{u(z)}{\sup\limits_{w \in E}\psi(w)}$ belongs to the class $\in \mathcal{U}(E,D)$.
Therefore,

\[
u(z)
\leq
-\sup_{w \in E}\psi(w)\,\omega_\alpha^*(z,E,D).
\]
Since $u\in \mathcal{U}(E,D,\psi)$ is arbitrary, we get
\[
\omega_\alpha^*(z,E,D,\psi)
\leq
-\sup_{w \in E}\psi(w)\,\omega_\alpha^*(z,E,D),
\qquad \forall z \in D.
\]
The proof is complete.
\end{proof}

Combining Proposition~\ref{prop: properties of alpha-measure} with Lemma~\ref{connection inequality}, we obtain the following result.

\begin{corollary}
The weighted $\alpha$-subharmonic measure $\omega_\alpha^*(z,E,D,\psi)$ is identically equal to zero if and only if $E$ is an $\alpha$-polar set, i.e., there exists a function $\rho\in \alpha\text{-}sh(D)$ such that $\rho|_E=-\infty$ and $\rho\not\equiv -\infty$. 
\end{corollary}

Let $E$ be an arbitrary subset of $D$.

\begin{proposition}\label{alpha-harmonic}
The weighted $\alpha$-subharmonic measure
$\omega_{\alpha}^*(z,E,D,\psi)$ is $\alpha$-harmonic in
$D\setminus \overline{E}$.
\end{proposition}

\begin{proof}
We assume that $D\setminus \overline{E}\neq \varnothing$.  By the technical Choquet lemma (see, for example, \cite{AS}), there exists
a sequence $\{u_j\}\subset \mathcal{U}(E,D,\psi)$ such that
\[
\left(\sup_{j\geq 1} u_j(z)\right)^*
=
\omega_{\alpha}^*(z,E,D,\psi).
\]
For each $m\in\mathbb{N}$, put
\[
v_m(z)=\max\{u_1(z),u_2(z),\ldots,u_m(z)\}.
\]
Then $v_m\in \mathcal{U}(E,D,\psi)$ and the sequence $\{v_m\}$ increases
pointwise to $\sup_{j\geq 1}u_j(z)$.

Let $B\Subset D\setminus \overline{E}$ be a ball. For each $m$, define
\[
w_m(z)=
\begin{cases}
\displaystyle
\int_{\partial B} v_m(\xi)P_{\alpha}(z,\xi)\,d\sigma(\xi),
& z\in B,\\[2ex]
v_m(z), & z\in D\setminus B,
\end{cases}
\]
where $P_{\alpha}(z,\xi)$ denotes the corresponding Poisson kernel in $B$.
By the submean value property for $\alpha$-subharmonic functions, we have
\[
v_m(z)
\leq
\int_{\partial B} v_m(\xi)P_{\alpha}(z,\xi)\,d\sigma(\xi)
=
w_m(z),
\qquad z\in B.
\]

By the construction of $w_m$, it is $\alpha$-harmonic in $B$ and $\alpha$-subharmonic in $D$. In addition, it is easy to verify the fact $w_m\in \mathcal{U}(E,D,\psi)$.

Moreover, since $\{v_m\}$ is increasing, the sequence $\{w_m\}$ is also
increasing. Thus, Levi's theorem
(see \cite{KK}, Theorem~2, p.~58) gives, for every $z\in B$,
\[
\lim_{m\to\infty}w_m(z)
=
\lim_{m\to\infty}
\int_{\partial B} v_m(\xi)P_{\alpha}(z,\xi)\,d\sigma(\xi)
=
\int_{\partial B}
\lim_{m\to\infty}v_m(\xi)P_{\alpha}(z,\xi)\,d\sigma(\xi).
\]
Denote
\[
w(z):=\lim_{m\to\infty}w_m(z).
\]
Then
\[
w(z)
=
\int_{\partial B}w(\xi)P_{\alpha}(z,\xi)\,d\sigma(\xi),
\qquad z\in B,
\]
and hence $w$ is $\alpha$-harmonic in $B$.
From $v_m\leq w_m$ in $B$, we obtain
\[
\sup_{j\geq 1}u_j(z)
=
\lim_{m\to\infty}v_m(z)
\leq
\lim_{m\to\infty}w_m(z)
=
w(z).
\]
Since $w_m\in \mathcal{U}(E,D,\psi)$ for every $m$, by the definition of
$\omega_{\alpha}^*(z,E,D,\psi)$, we also have
\[
w_m(z)\leq \omega_{\alpha}^*(z,E,D,\psi)
=
\left(\sup_{j\geq 1}u_j(z)\right)^* \le w(z), \qquad \forall z \in B.
\]
Passing to the limit as $m\to\infty$, it follows that
\[
w(z)=
\omega_{\alpha}^*(z,E,D,\psi), \qquad \forall z \in B. 
\]

Thus, $\omega_{\alpha}^*(z,E,D,\psi)$ is $\alpha$-harmonic in $B$.
Since $B\Subset D\setminus \overline{E}$ was arbitrary, the assertion follows. The proof is complete.
\end{proof}

Proposition~\ref{alpha-harmonic} immediately yields the following result in the unweighted case.

\begin{corollary}
The $\alpha$-subharmonic measure $\omega_{\alpha}^*(z,E,D)$ is
$\alpha$-harmonic in $D\setminus \overline{E}$.
\end{corollary}

The following proposition establishes convergence properties of the weighted
$\alpha$-subharmonic measure with respect to decreasing sequences of open
neighborhoods and increasing compact exhaustions.

\begin{proposition}
The following relations hold.
\begin{enumerate}

\item Let $E\subset D$ be an arbitrary set and let $\psi$ be lower semicontinuous on $E$. Then, there exists a
decreasing sequence of open sets $U_j\supset E$ such that
\[
\left(\lim_{j\to\infty}\omega_{\alpha}^{*}(z,U_j,D,\widetilde{\psi})\right)^{*}
=
\omega_{\alpha}^{*}(z,E,D,\psi),
\]
where $\widetilde{\psi}$ is lower semicontinuous in some open neighborhood
$V$ of $E$ and satisfies
\[
\widetilde{\psi}|_E=\psi|_E, \qquad
\sup_{z\in V}\widetilde{\psi}(z)<0.
\]

\item \label{item: second} Let $U\subset D$ be an open set and let
$U=\bigcup_{j=1}^{\infty}K_j,$
where $K_j\subset K_{j+1}^{\circ}$ are compact sets. Assume that $\psi$ is
upper semicontinuous in $U$. Then
\[
\omega_{\alpha}^{*}(z,K_j,D,\psi)
\downarrow
\omega_{\alpha}^{*}(z,U,D,\psi).
\]
\end{enumerate}
\end{proposition}

\begin{proof}
We prove the first assertion.  By Choquet's lemma (see, for example, \cite{AS}),
there exists a countable family
\[
\{u_k\}\subset\mathcal{U}(E,D,\psi)
\]
such that
\[
\left(\sup_k u_k(z)\right)^{*}
=
\omega_{\alpha}^{*}(z,E,D,\psi).
\]
Define a sequence $\{v_j\}$ as follows:
\[
v_j(z)=\max\{u_1(z),u_2(z),\ldots,u_j(z)\}.
\]
Then the sequence $\{v_j\}$ is increasing and
\[
\left(\lim_{j\to\infty}v_j(z)\right)^{*}
=
\omega_{\alpha}^{*}(z,E,D,\psi).
\]

Since $\psi$ is lower semicontinuous on $E$, the function $-\psi$ is upper semicontinuous on $E$. It is well known that every upper semicontinuous function defined on a set can be extended as an upper semicontinuous function to some neighborhood of that set.

For each $j\in\mathbb{N}$, consider the following open sets
\[
U_j=\left\{z\in V: v_j(z)-\widetilde{\psi}(z)<\frac{1}{j}\right\},
\]
where $V\subset D$ is an open neighborhood of $E$, the function
$-\widetilde{\psi}$ is upper semicontinuous on $V$, and
\[
\widetilde{\psi}|_E=\psi|_E,  \qquad \sup_{z \in V}\widetilde{\psi}(z)<0.
\]
By construction, for all $j \in \mathbb{N}$,
$$E\subset U_j
\quad
U_{j+1}\subset U_j, \quad  \text{and}\quad v_j-\frac{1}{j}\in\mathcal{U}(U_j,D,\widetilde{\psi}) $$
Hence,
$v_j(z)-\frac{1}{j}
\leq
\omega_{\alpha}^{*}(z,U_j,D,\widetilde{\psi})$
and by the monotonicity property,
\[
v_j(z)-\frac{1}{j}
\leq
\omega_{\alpha}^{*}(z,U_j,D,\widetilde{\psi})
\leq
\omega_{\alpha}^{*}(z,E,D,\widetilde{\psi}).
\]
Letting $j\to\infty$ and taking the upper regularization, we get
\[
\left(\lim_{j\to\infty}\omega_{\alpha}^{*}(z,U_j,D,\widetilde{\psi})\right)^{*}
=
\omega_{\alpha}^{*}(z,E,D,\widetilde{\psi}).
\]

We next prove assertion~\ref{item: second}. Since $K_j\subset K_{j+1}$, the monotonicity of the weighted $\alpha$-subharmonic measure implies that the sequence
\[
\left\{\omega_{\alpha}^{*}(z,K_j,D,\psi)\right\}_{j=1}^{\infty}
\]
is decreasing. Moreover, for every $z\in D$ and every $j\in\mathbb{N}$, we have
\[
\omega_{\alpha}^{*}(z,K_j,D,\psi)
\geq
\omega_{\alpha}^{*}(z,U,D,\psi).
\]
Hence, if we define
\[
\omega_{\alpha}(z)
:=
\lim_{j\to\infty}\omega_{\alpha}^{*}(z,K_j,D,\psi),
\]
then $\omega_{\alpha}\in\alpha\text{-}sh(D)$, being the decreasing limit of $\alpha$-subharmonic functions, and
\begin{equation} \label{first inequality}
\omega_{\alpha}(z)\geq \omega_{\alpha}^{*}(z,U,D,\psi), \quad \forall z \in D.
\end{equation}

It remains to prove the reverse inequality. Let $z\in U$. Since
\[
U=\bigcup_{j=1}^{\infty}K_j
\quad\text{and}\quad
K_j\subset K_{j+1}^{\circ},
\]
there exists $j_0$ such that $z\in K_j$ for all $j\geq j_0$. Hence, by the definition of $\omega_{\alpha}^{*}(z,K_j,D,\psi)$ and the upper
semicontinuity of $\psi$ on $U$, we have $\omega_{\alpha}(z)\leq \psi(z)$
for every $z\in U$.\\
Consequently, $\omega_{\alpha}\in \mathcal{U}(U,D,\psi)$  and thus,

\begin{equation}\label{second inequality}
\omega_{\alpha}(z)
\leq
\omega_{\alpha}^{*}(z,U,D,\psi), \quad \forall z \in D.
\end{equation}
Combining inequalities \eqref{first inequality} and \eqref{second inequality}, we conclude that
\[
\omega_{\alpha}^{*}(z,K_j,D,\psi)
\downarrow
\omega_{\alpha}^{*}(z,U,D,\psi).
\]
The proof is complete.
\end{proof}

We next state a theorem which extends Proposition~\ref{Pro sets} and will play an important role in the subsequent sections.

\begin{proposition}\label{extends}
The following properties hold for the weighted $\alpha$-subharmonic measure.
\begin{enumerate}

\item \label{item:1}
Let $A \subset D$ be an $\alpha$-polar set. Then, for any set $E\subset D$,
\[
\omega_{\alpha}^*(z,E,D,\psi)
=
\omega_{\alpha}^*(z,E \cup A,D,\psi)
\]
for all $z\in D$, where $\psi$ is defined on $E \cup A$.

\item \label{item:2}
If $E\subset \subset D$, then
\[
\lim_{z\to \partial D,\; z\in D}
\omega_\alpha^*(z,E,D,\psi)=0.
\]

\item \label{item:3}
Let $\{K_j\}$ be a sequence of compact subsets of $D$ such that
$K_{j+1}\subset K_j$ and let $\psi$ be lower semicontinuous in  $K_1$. Then
\[
\lim_{j\to\infty}\omega_{\alpha}(z,K_j,D,\psi)
=
\omega_{\alpha}(z,\bigcap_{j=1}^{\infty}K_j,D,\psi)
\]
for all $z\in D$.

\item \label{item:4}
Let $E \subset \subset D_1$ and let $\{D_j\}$ be an increasing sequence of domains such that \\
$D_j\subset D_{j+1},
\quad
D=\bigcup_{j=1}^{\infty}D_j.$
Then
\[
\lim_{j\to\infty}
\omega_{\alpha}^{*}(z,E,D_j,\psi)
=
\omega_{\alpha}^{*}(z,E,D,\psi).
\]
\end{enumerate}
\end{proposition}

\begin{proof}
We first prove \ref{item:1}. Since $E\subset E\cup A$, the monotonicity property gives
\[
\omega_{\alpha}^*(z,E,D,\psi)
\geq
\omega_{\alpha}^*(z,E\cup A,D,\psi),
\qquad z\in D.
\]
Conversely, since $A$ is $\alpha$-polar in $D$, there exists a function
$v\in \alpha\text{-}sh(D)$, $v\not\equiv -\infty$, such that
$v|_A=-\infty$
(see \cite{ABSR}). Take an arbitrary function
$u\in \mathcal{U}(E,D,\psi)$ and fix $\varepsilon>0$. Then
\[
u+\varepsilon v\in \alpha\text{-}sh(D), \qquad u+\varepsilon v<0
\quad \text{in } D
\]
and
\[
(u+\varepsilon v)|_{E\cup A}\leq \psi|_{E\cup A}.
\]
Hence
\[
u+\varepsilon v\in \mathcal{U}(E\cup A,D,\psi)
\]
and
\[
u(z)+\varepsilon v(z)
\leq
\omega_{\alpha}(z,E\cup A,D,\psi),
\qquad \forall z\in D.
\]
Taking the supremum over all $u\in\mathcal{U}(E,D,\psi)$ and then letting
$\varepsilon\to 0$, we obtain
\[
\omega_{\alpha}^*(z,E,D,\psi)
\leq
\omega_{\alpha}^*(z,E\cup A,D,\psi).
\]

We now prove \ref{item:2}. Since $D$ is an $\alpha$-regular domain, there exists
a function $\rho\in \alpha\text{-}sh(D)$ such that
\[
\rho<0 \quad \text{in } D,
\qquad
\lim_{z\to\partial D}\rho(z)=0.
\]
Then $C\rho\in \mathcal{U}(E,D,\psi)$, where $ C=\frac{\inf\limits_{z\in E}\psi(z)}
        {\max\limits_{z\in \overline{E}}\rho(z)}$ and
\[
C\rho(z)
\leq
\omega_{\alpha}^*(z,E,D,\psi)
\leq 0,
\qquad \forall z\in D.
\]
Since $\rho(z)\to 0$ as $z\to\partial D$, it follows that
\[
\lim_{z\to \partial D,\; z\in D}
\omega_{\alpha}^*(z,E,D,\psi)=0.
\]

Next, we prove \ref{item:3}. Fix an arbitrary point $z^0\in D$ and an arbitrary number $\varepsilon>0$.
By the definition of $\omega_{\alpha}(z^0,K,D,\psi)$ with $K=\bigcap\limits_{j=1}^{\infty}K_j$, there exists
$u\in \mathcal{U}(K,D,\psi)$ such that
\[
\omega_{\alpha}(z^0,K,D,\psi)-\varepsilon<u(z^0).
\]
Since $\psi$ is lower semicontinuous on $K_1$, there exist a neighbourhood 
$V \subset D$ of $K_1$ and a lower semicontinuous function 
$\widetilde{\psi}$ on $V$ such that $\widetilde{\psi}|_{K_1}=\psi|_{K_1}$. 
Then the function $u-\widetilde{\psi}$ is upper semicontinuous on $V$. Hence, the set
\[
U=\{z\in V: u(z)<\widetilde{\psi}(z)+\varepsilon\}
\]
is open and contains $K$. Since $K_j\downarrow K$, there exists
$j_0\in\mathbb{N}$ such that
\[
K_j\subset U
\qquad \text{for all } j\geq j_0.
\]
Thus,
\[
u-\varepsilon\in \mathcal{U}(K_j,D,\widetilde{\psi})
\qquad \text{for all } j\geq j_0.
\]
Hence
\[
u(z)-\varepsilon
\leq
\omega_{\alpha}(z,K_j,D,\widetilde{\psi}),
\quad \forall z\in D,\quad \forall j\geq j_0.
\]
In particular, at $z=z^0$, we obtain
\[
\omega_{\alpha}(z^0,K,D,\widetilde{\psi})-2\varepsilon
\leq
\omega_{\alpha}(z^0,K_j,D,\widetilde{\psi}),
\quad \forall j\geq j_0.
\]
On the other hand, since $K\subset K_j$, the monotonicity property gives
\[
\omega_{\alpha}(z^0,K_j,D,\widetilde{\psi})
\leq
\omega_{\alpha}(z^0,K,D,\widetilde{\psi}).
\]
Therefore,
\[
\omega_{\alpha}(z^0,K,D,\widetilde{\psi})-2\varepsilon
\leq
\omega_{\alpha}(z^0,K_j,D,\widetilde{\psi})
\leq
\omega_{\alpha}(z^0,K,D,\widetilde{\psi}),
\quad \forall j\geq j_0.
\]
Since $z^0\in D$ and $\varepsilon>0$ were arbitrary, the assertion follows.

Finally, we prove \ref{item:4}. By the monotonicity property,
\[
\omega_{\alpha}^*(z,E,D_j,\psi)
\geq
\omega_{\alpha}^*(z,E,D_{j+1},\psi).
\]
Hence
\[
\omega_{\alpha}(z)
:=
\lim_{j\to\infty}\omega_{\alpha}^*(z,E,D_j,\psi)
\]
is $\alpha$-subharmonic in $D$ and for all $z \in D,$
\[
\omega_{\alpha}(z)
\geq
\omega_{\alpha}^*(z,E,D,\psi).
\]

Conversely, since $D$ is $\alpha$-regular, there exists a function
$\rho\in \alpha\text{-}sh(D)$ such that
\[
\rho<0 \quad \text{in } D,
\qquad
\lim_{z\to\partial D}\rho(z)=0.
\]
Choose a sequence $\varepsilon_j>0$ with $\varepsilon_j\to 0$ such that
\[
\{z\in D:\rho(z)\leq -\varepsilon_j\}\subset D_j.
\]
Take an arbitrary function $u\in \mathcal{U}(E,D_j,\psi)$ and define
\[
v(z)=
\begin{cases}
\max\{u(z)-C\varepsilon_j,\; C\rho(z)\}, & z\in D_j,\\[2mm]
C\rho(z), & z\in D\setminus D_j,
\end{cases}
\]
where
\[
C=
\frac{\inf\limits_{z\in E}\psi(z)}
     {\max\limits_{z\in \overline{E}}\rho(z)}.
\]
Then $v\in \mathcal{U}(E,D,\psi)$. Hence, for all $z\in D$,
\[
v(z)\leq \omega_{\alpha}^*(z,E,D,\psi).
\]
In particular, for any $z\in D_j$, we have the following inequality
\[
u(z)-C\varepsilon_j
\leq
v(z)
\leq
\omega_{\alpha}^*(z,E,D,\psi).
\]
Since $u\in \mathcal{U}(E,D_j,\psi)$ was chosen arbitrarily, we obtain
\[
\omega_{\alpha}^*(z,E,D_j,\psi)-C\varepsilon_j
\leq
\omega_{\alpha}^*(z,E,D,\psi),
\quad \forall z\in D_j.
\]
Passing to the limit as $j\to\infty$, we get
\[
\lim_{j\to\infty}
\omega_{\alpha}^*(z,E,D_j,\psi)
\leq
\omega_{\alpha}^*(z,E,D,\psi).
\]
This completes the proof.
\end{proof}

\section{$(\alpha,\psi)$-regularity of compacts} \label{regularity of compacts}

In this section, we present the proofs of Theorem~\ref{weighet continuity} and Theorem~\ref{thm: regularity}.
Assume that the function $\psi$ admits an $\alpha$-subharmonic extension to $D$, that is, there exists a function
\begin{equation}\label{alpha-subharmonic extension}
\widetilde{\psi} \in \alpha\text{-}sh(D), 
\qquad 
\widetilde{\psi}|_E=\psi|_E, 
\qquad 
\widetilde{\psi}<0 \quad \text{in } D.
\end{equation}
Then, by the definition of $\omega_{\alpha}(z,E,D,\psi)$, we have
\[
\omega_{\alpha}(z,E,D,\psi) \geq \widetilde{\psi}(z),
\qquad z\in D.
\]
In particular, since every function $u\in\mathcal{U}(E,D,\psi)$ satisfies $u|_E\leq \psi|_E$, it follows that
\begin{equation}\label{omega=psi in E}
\omega_{\alpha}(z,E,D,\psi)=\psi(z),
\qquad z\in E.
\end{equation}

 Throughout the rest of the paper, we assume that condition~\eqref{omega=psi in E} holds in the definition of
$\omega_{\alpha}(z,E,D,\psi)$.

 Let us recall the definition of $(\alpha,\psi)$-regularity. Let $D$ be an
$\alpha$-regular domain, and let $K\subset D$ be a compact set.

\begin{definition}\label{weighet reg}
A point $z_0\in K$ is called an \textit{$(\alpha,\psi)$-regular point of $K$ with respect to $D$} if
\[
\omega_{\alpha}^*(z_0,K,D,\psi)=\psi(z_0).
\]
It is called a \textit{locally $(\alpha,\psi)$-regular point of $K$ with respect to $D$} if, for every
neighborhood $B\subset \mathbb{C}^n$ of $z_0$, one has
\[
\omega_{\alpha}^*(z_0,K\cap \overline{B},D,\psi)=\psi(z_0).
\]
If every point of $K$ is $(\alpha,\psi)$-regular, then $K$ is called an
\textit{$(\alpha,\psi)$-regular compact set with respect to $D$}. If every point of $K$ is
locally $(\alpha,\psi)$-regular, then $K$ is called a
\textit{locally $(\alpha,\psi)$-regular compact set with respect to $D$}.
\end{definition}

We also note that, since $D$ is fixed throughout the paper, for convenience,
we omit the phrase “with respect to $D$” and simply say
$(\alpha,\psi)$-regular, respectively locally $(\alpha,\psi)$-regular,
compact sets.

It follows from Definition~\ref{weighet reg} and Lemma~\ref{prop: monotonicity} that local $(\alpha,\psi)$-regularity implies $(\alpha,\psi)$-regularity. The converse does not hold in general.

Our next goal is to prove Theorem~\ref{weighet continuity}. This theorem may be
viewed as a weighted version of Theorem~\ref{continuity}. It shows that
$(\alpha,\psi)$-regularity can be described by the continuity of the associated
weighted $\alpha$-subharmonic measure. Moreover, the weighted $\alpha$-subharmonic
measure is interpreted as a solution to the corresponding elliptic boundary value
problem.

\begin{proof}[proof of theorem \ref{weighet continuity}]
Let $K$ be a compact subset of $D$. Assume first that the function $\omega_{\alpha}(z,K,D,\psi)$ is continuous in $D$.
Then, by condition~\eqref{omega=psi in E}, we have
\[
\omega_{\alpha}(z,K,D,\psi)=\psi(z), \qquad z\in K.
\]
Hence, $\psi$ is continuous on $K$. Moreover, since
$\omega_{\alpha}(z,K,D,\psi)$ is continuous in $D$, we have
\[
\omega_{\alpha}(z,K,D,\psi) \equiv \omega_{\alpha}^{*}(z,K,D,\psi).
\]
Therefore,
\[
\omega_{\alpha}^{*}(z,K,D,\psi)=\psi(z),
\qquad z\in K.
\]
This means that $K$ is $(\alpha,\psi)$-regular.

Conversely, assume that $K$ is $(\alpha,\psi)$-regular. Then
\[
\omega_{\alpha}^{*}(z,K,D,\psi)=\psi(z),
\qquad z\in K.
\]
Hence, the function $\omega_{\alpha}^{*}(z,K,D,\psi)$ is an element of the class $\mathcal{U}(K,D,\psi)$ and  we find
\[
\omega_{\alpha}^{*}(z,K,D,\psi)
=
\omega_{\alpha}(z,K,D,\psi),
\qquad z\in D.
\]
It remains to prove that $\omega_{\alpha}^{*}(z,K,D,\psi)$ is continuous in $D$.
Since $\psi<0$ on $K$ and $\psi\in C(K)$, we may choose $\varepsilon>0$ such that
\[
\psi(z)<-\varepsilon, \qquad z\in K.
\]
By Proposition~\ref{extends}, the set
\[
G_{\varepsilon}
=
\{z\in D:\omega_{\alpha}^{*}(z,K,D,\psi)<-\varepsilon\}
\]
satisfies
$K\subset G_{\varepsilon}\Subset D.$
Since $\omega_{\alpha}^{*}(z,K,D,\psi)$ is $\alpha$-subharmonic in $D$, there exists a sequence of nonpositive functions
\[
u_j\in \alpha\text{-}sh(G)\cap C^{\infty}(G)
\]
such that
\[
\lim_{j \to \infty} u_j(z) =\omega_{\alpha}^{*}(z,K,D,\psi)
\]
for every $z\in G$, where
$G_{\varepsilon}\Subset G\Subset D$ (see \cite{ABIS}).

On the other hand, since $\psi\in C(K)$, by Whitney's extension theorem
(see \cite{WH}), there exists a continuous function $\widetilde{\psi}$ in $D$
such that
$\widetilde{\psi}|_K=\psi|_K.$
Consider the open set
\[
U_{\varepsilon}
=
\{z\in D:
\omega_{\alpha}^{*}(z,K,D,\psi)<\widetilde{\psi}(z)+\varepsilon\}.
\]
Clearly, $K\subset U_{\varepsilon}$. Applying Lemma~\ref{Hartogs} to the compact set $K$ and the open set
$U_{\varepsilon}$, we obtain $j_0\in\mathbb{N}$ such that
\[
u_j(z)<\psi(z)+\varepsilon,
\quad \forall z\in K,\quad \forall j\geq j_0.
\]
For any $j\geq j_0$, define
\[
v_j(z)=
\begin{cases}
\max\{u_j(z)-\varepsilon,\,
\omega_{\alpha}^{*}(z,K,D,\psi)\},
& z\in G_{\varepsilon},\\[2mm]
\omega_{\alpha}^{*}(z,K,D,\psi),
& z\in D\setminus G_{\varepsilon}.
\end{cases}
\]
By construction, $v_j\in \mathcal{U}(K,D,\psi)$; hence
\[
v_j(z)\leq \omega_{\alpha}^{*}(z,K,D,\psi),
\quad \forall z\in D, \quad \forall j\geq j_0 .
\]
In particular, for $z\in G_{\varepsilon}$ and $j\geq j_0$,
\[
u_j(z)-\varepsilon
\leq
\omega_{\alpha}^{*}(z,K,D,\psi).
\]
Fix an arbitrary point $z^0\in G_{\varepsilon}$. Taking the lower limit as
$z\to z^0$ in the last inequality and using the continuity of $u_j$, we obtain
\[
u_j(z^0)-\varepsilon
\leq
\liminf_{z\to z^0}
\omega_{\alpha}^{*}(z,K,D,\psi),
\qquad \forall j\geq j_0.
\]
Letting $j\to\infty$, we get
\[
\omega_{\alpha}^{*}(z^0,K,D,\psi)-\varepsilon
\leq
\liminf_{z\to z^0}
\omega_{\alpha}^{*}(z,K,D,\psi).
\]
Since $\varepsilon>0$ is arbitrary, it follows that
\[
\omega_{\alpha}^{*}(z^0,K,D,\psi)
\leq
\liminf_{z\to z^0}
\omega_{\alpha}^{*}(z,K,D,\psi).
\]
Thus, $\omega_{\alpha}^{*}(z,K,D,\psi)$ is lower semicontinuous at $z^0$.
Since $\omega_{\alpha}^{*}(z,K,D,\psi)$ is upper semicontinuous by definition, it is continuous at $z^0$. Since $z^0\in G_{\varepsilon}$ was arbitrary, $\omega_{\alpha}^{*}(z,K,D,\psi)$ is continuous in $G_{\varepsilon}$. Consequently, it is continuous in $D$.

Assume now that $K$ is an $(\alpha,\psi)$-regular compact set. We prove that
$\omega_{\alpha}^*(z,K,D,\psi)$ is the unique $\alpha$-subharmonic function in
$D$ satisfying the boundary value problem \eqref{HO}. By
Proposition~\ref{alpha-harmonic}, the function
$\omega_{\alpha}(z,K,D,\psi)$ satisfies the equation
$\Delta_{\alpha}u=0$ in $D\setminus K$. Moreover,
property~\ref{item:2} of Proposition~\ref{extends} gives the boundary condition
on $\partial D$, while the $(\alpha,\psi)$-regularity of $K$ gives the condition
on $K$. Hence, $\omega_{\alpha}(z,K,D,\psi)$ satisfies all the conditions in
\eqref{HO}. The uniqueness follows from the maximum principle in the class of
$\alpha$-subharmonic functions.
The proof is complete.
\end{proof}

By property~\ref{item:2} of Proposition~\ref{extends}, if we define
\[
\omega_{\alpha}^{*}(z,K,D,\psi)=0, \qquad z\in \partial D,
\]
then the continuity of $\omega_{\alpha}^{*}(z,K,D,\psi)$ in $D$ implies its continuous extendability to $\overline{D}$.

We next turn to the proof of Theorem~\ref{thm: regularity}.

\begin{proof}[proof of theorem \ref{thm: regularity}] We first prove the first assertion. We argue by showing the equivalence of the
negations of the two notions of local regularity.

Assume first that \(z_0\in K\) is not locally \((\alpha,\psi)\)-regular. Then there
exist a ball \(B\Subset D\), with \(z_0\in B\), and a number \(\lambda>0\) such that
\[
\omega_{\alpha}^{*}(z_0,K\cap \overline{B},D,\psi)
\geq \psi(z_0)+\lambda .
\]
By the monotonicity property, for every ball \(B_1\Subset B\) such that
\(z_0\in B_1\), we also have
\[
\omega_{\alpha}^{*}(z_0,K\cap \overline{B}_1,D,\psi)
\geq \psi(z_0)+\lambda .
\]
Combining this with Lemma~\ref{connection inequality}, we obtain
\[
\psi(z_0)+\lambda
\leq
\omega_{\alpha}^{*}(z_0,K\cap \overline{B}_1,D,\psi)
\leq
-\sup_{z\in K\cap \overline{B}_1}\psi(z)\,
\omega_{\alpha}^{*}(z_0,K\cap \overline{B}_1,D).
\]
Since \(\psi\) is continuous on \(K\), after decreasing \(B_1\), if necessary, we may assume that
\[
\sup_{z\in K\cap \overline{B}_1}\psi(z)<\psi(z_0)+\lambda<0 .
\]
Hence,

\[
\omega_{\alpha}^{*}(z_0,K\cap \overline{B}_1,D)>-1 
\]
and thus \(z_0\) is not locally \(\alpha\)-regular.

Conversely, suppose that \(z_0\in K\) is not locally \(\alpha\)-regular. Then there
exist a ball \(B\Subset D\), with \(z_0\in B\), and a number
\(0<\varepsilon<1\) such that
\[
\omega_{\alpha}^{*}(z_0,K\cap \overline{B},D)\geq -1+\varepsilon .
\]
Again, by monotonicity, for every ball \(B_1\Subset B\) with \(z_0\in B_1\), we have
\[
\omega_{\alpha}^{*}(z_0,K\cap \overline{B}_1,D)\geq -1+\varepsilon .
\]
Using Lemma~\ref{connection inequality}, we obtain
\[
\omega_{\alpha}^{*}(z_0,K\cap \overline{B}_1,D,\psi)
\geq
-\inf_{z\in K\cap \overline{B}_1}\psi(z)\,
\omega_{\alpha}^{*}(z_0,K\cap \overline{B}_1,D).
\]
Therefore,
\[
\omega_{\alpha}^{*}(z_0,K\cap \overline{B}_1,D,\psi)
\geq
\inf_{z\in K\cap \overline{B}_1}\psi(z)(1-\varepsilon).
\]
Since \(\psi\) is continuous on \(K\), choosing \(B_1\) sufficiently small and
\(\varepsilon>0\) sufficiently small, we may ensure that
\[
\inf_{z\in K\cap \overline{B}_1}\psi(z)>
\frac{\psi(z_0)}{1-\varepsilon}.
\]
Consequently,
\[
\omega_{\alpha}^{*}(z_0,K\cap \overline{B}_1,D,\psi)
>
\psi(z_0).
\]
This shows that \(z_0\) is not locally \((\alpha,\psi)\)-regular. Hence, the first
assertion is proved.

We now prove the second part of the theorem. It is clear that every locally
\((\alpha,\psi)\)-regular point is globally \((\alpha,\psi)\)-regular. We shall prove
the converse implication.
Assume that \(z^0\in K\) is globally \((\alpha,\psi)\)-regular. Then
\[
\omega_{\alpha}^*(z^0,K,D,\psi)=\psi(z^0).
\]
By the assumption of the theorem, the function \(\psi\) admits an extension
\(\widetilde{\psi}\) which is strictly \(\alpha\)-subharmonic in a domain
\(D^+\supset \overline{D}\) and satisfies
\[
\widetilde{\psi}|_K=\psi|_K,
\qquad
\widetilde{\psi}|_D<0.
\]
Therefore, there exists a constant \(\varepsilon>0\) such that the function
\[
\widetilde{\psi}(z)-\varepsilon |z-z^0|^2
\]
is \(\alpha\)-subharmonic in \(D\).
Choose \(r>0\) such that
\[
B_r=B(z^0,r)\Subset D
\]
and
\[
\max_{z\in \overline{B}_r}\widetilde{\psi}(z)+\varepsilon r^2<0.
\]
Let \(u\in \mathcal{U}(K\cap \overline{B}_r,D)\) be arbitrary and define
\[
\varphi(z)
=
\varepsilon r^2\bigl(u(z)+1\bigr)
+
\widetilde{\psi}(z)
-
\varepsilon |z-z^0|^2 .
\]
Then \(\varphi\) is \(\alpha\)-subharmonic in $D.$
We claim that \(\varphi\in \mathcal{U}(K,D,\psi)\). Indeed, if
\(z\in K\cap \overline{B}_r\), then \(u(z)\leq -1\), and hence
\[
\varphi(z)\leq \widetilde{\psi}(z)=\psi(z).
\]
On the other hand, if \(z\in K\setminus \overline{B}_r\), then
\(|z-z^0|\geq r\). Since \(u<0\) in \(D\), we have
\[
\varepsilon r^2\bigl(u(z)+1\bigr)-\varepsilon |z-z^0|^2\leq 0.
\]
Therefore,
\[
\varphi(z)\leq \widetilde{\psi}(z)=\psi(z),
\qquad z\in K.
\]
Moreover, by the choice of \(r\), we have \(\varphi<0\) in \(D\). Thus
\[
\varphi\in \mathcal{U}(K,D,\psi).
\]
Consequently,
\[
\varphi(z)\leq \omega_{\alpha}^*(z,K,D,\psi),
\quad \forall z\in D.
\]
Since \(u\in \mathcal{U}(K\cap \overline{B}_r,D)\) was arbitrary, it follows that
\[
\varepsilon r^2
\left(
\omega_{\alpha}^*(z,K\cap \overline{B}_r,D)+1
\right)
+
\widetilde{\psi}(z)
-
\varepsilon |z-z^0|^2
\leq
\omega_{\alpha}^*(z,K,D,\psi),
\quad \forall z\in D.
\]
Putting \(z=z^0\), we obtain
\[
\varepsilon r^2
\left(
\omega_{\alpha}^*(z^0,K\cap \overline{B}_r,D)+1
\right)
+
\psi(z^0)
\leq
\omega_{\alpha}^*(z^0,K,D,\psi)
=
\psi(z^0).
\]
It follows that
\[
\omega_{\alpha}^*(z^0,K\cap \overline{B}_r,D)\leq -1.
\]
On the other hand, by Definition~\ref{def of asm}, we have
\[
\omega_{\alpha}^*(z^0,K\cap \overline{B}_r,D)\geq -1.
\]
Consequently,
\[
\omega_{\alpha}^*(z^0,K\cap \overline{B}_r,D)=-1.
\]
Hence \(z^0\) is locally \(\alpha\)-regular. Therefore, by the first part of
the theorem, \(z^0\) is locally \((\alpha,\psi)\)-regular. 
The proof is complete.
\end{proof}
 
By Theorem~\ref{thm: regularity}, we obtain the following important corollaries.

\begin{corollary}
Let $K \subset D$ be a compact set. Suppose that $K$ is globally
$(\alpha,\psi)$-regular, where $\psi \in C(K)$ can be extended to a strictly
$\alpha$-subharmonic function in some neighborhood $D^+ \supset \overline{D}$
such that the extension belongs to $\mathcal{U}(K,D,\psi)$. Then $K$ is locally
$\alpha$-regular.
\end{corollary}

\begin{corollary}
Let $\psi_1,\psi_2 \in C(K)$, and suppose that each $\psi_j$, $j=1,2$, can be
extended to a strictly $\alpha$-subharmonic function in some neighborhood
$D^+ \supset \overline{D}$ such that the extension belongs to
$\mathcal{U}(K,D,\psi_j)$. Then a point $z^0 \in K \subset D$ is
$(\alpha,\psi_1)$-regular if and only if it is $(\alpha,\psi_2)$-regular.
\end{corollary}

\begin{corollary}
Let $K \subset D$ be a compact set. Suppose that $K$ is globally
$(\alpha,\psi)$-regular, where $\psi \in C(K)$ can be extended to a strictly
$\alpha$-subharmonic function in some neighborhood $D^+ \supset \overline{D}$
such that the extension belongs to $\mathcal{U}(K,D,\psi)$. Then, for every
$z^0 \in K$ and every neighborhood $B \subset D$ of $z^0$, the set
$E=B\cap K$ is not $\alpha$-polar.
\end{corollary}

\section{Hölder continuity of the weighted $\alpha$-subharmonic measure} \label{Holder section}

In this section, by adapting the method used in \cite{KKRK}, we prove
Theorem~\ref{Holder}, which establishes the H\"older continuity of the
weighted $\alpha$-subharmonic measure.

\begin{proof}[Proof of Theorem~\ref{Holder}]
The H\"older continuity of $\omega_{\alpha}(z,K,D,\psi)$ in $D$ immediately
implies the H\"older continuity of $\psi$ on $K$ and condition
\eqref{Holdercondition}. Therefore, we prove the converse implication. Namely,
we show that the H\"older continuity of $\psi$ on $K$, together with condition
\eqref{Holdercondition}, implies the H\"older continuity of
$\omega_{\alpha}(z,K,D,\psi)$ in $D$.

Since \(D\) is strongly \(\alpha\)-regular,
there exist a neighborhood \(D^+\) of \(D\) and a function
\(\rho \in \alpha\text{-}sh(D^+) \cap C^2(D^+)\) such that
\[
D=\{ z \in D^+ : \rho(z)<0 \}.
\]
 In addition, $\psi(z)$ is negative in $K$ and the function $\omega_{\alpha}^*(z,K,D,\psi)$ can be extended from $D$ to $D^+$ as  an $\alpha$-subharmonic function, i.e., the following function
 \[
 \tilde \omega^*(z)= \begin{cases}\omega_{\alpha}^*(z,K,D,\psi), \, \, z \in D \\
 \rho(z), \, \, z \in D^+ \setminus D.
 \end{cases}
 \]
 is $\alpha$-subharmonic in $D^+$.  Condition~\eqref{Holdercondition} implies that
\[
\omega_{\alpha}^*(z,K,D,\psi)=\psi(z)
\]
for every \(z\in K\). Therefore, \(K\) is an \((\alpha,\psi)\)-regular compact set. Consequently, for all $z \in D$, we have 
 \[\omega_{\alpha}^*(z,K,D,\psi)=\omega_{\alpha}(z,K,D,\psi).\]
 Consider the following function
 \[
 \tilde\omega(z)= \begin{cases} \omega_{\alpha}(z,K,D,\psi), \, \, z \in D \\
 \rho(z), \, \, z \in D^+ \setminus D. \end{cases}
 \]
It is clear that, for every $z\in D^+$, we have
$\widetilde{\omega}(z)=\widetilde{\omega}^{*}(z),$
and that $\widetilde{\omega}$ is $\alpha$-subharmonic in $D^+$.
 Since $\rho\in C^2(D^+)$, it satisfies a Lipschitz condition on every domain
$G\Subset D^+$. That is, there exists a constant $L=L(G)>0$ such that, for all
$z',z''\in G$,
\[
|\rho(z')-\rho(z'')|\leq L|z'-z''|.
\]
We fix a domain $G$ such that
$D\Subset G\Subset D^+.$
  We take $\delta$ sufficiently small, so the relation
 \[
 0< \delta< \min \{dist(K,\partial D, dist(D,\partial G), 1 \}
 \]
 holds. Thus, it suffices to show that there exist constants $A >0$ and $\lambda \in (0,1]$ such that, for all $z', \, z'' \in D$ with $|z'-z''| \leq \delta$, the inequality
 \[
 |\omega_{\alpha}(z',K,D,\psi)-\omega_{\alpha}(z'',K,D,\psi)| \leq A |z'-z''|^{\lambda}
 \] holds.
Since $\psi(z)$ is H\"older continuous in $K$, there exist constants $C_1>0$ and $\lambda_1 \in (0,1]$ such that for any $ z ', \, z'' \in K$, 
\[
|\psi(z')-\psi(z'')| \leq C_1 |z'-z''|^{\lambda_1}
\]
is valid. 
Now, let's fix $t \in \mathbb{C}^n$ such that $|t| = \delta$. Then, the function $\tilde \omega(z+t)$ is $\alpha-$subharmonic for every point $z \in D$. Let $ w' \in K$ be a point such that 
\[|z+t-w'|=dist(z+t,K).\]
Thus, for $z \in K$, we have 
$|z+t-w'| \leq \delta,$
also, by the triangle inequality, we get 
$|z-w'| \leq 2 \delta.$
Since $|t| = \delta$ is valid, for all $z \in K$, $z+t \in D$. Therefore, for all $ z \in K$, we get 
\[
|\omega_{\alpha}(z+t,K,D,\psi)-\psi(z)|=|\omega_{\alpha}(z+t,K,D,\psi)-\psi(w')+\psi(w')-\psi(z)| \leq \]
\[
\leq |\omega_{\alpha}(z+t,K,D,\psi)-\psi(w')|+|\psi(w')-\psi(z)| \leq C_2 \delta^{\lambda_2}, 
\]
where $C_2=3\max\{C,C_1\}$ and $\lambda_2=\min\{\lambda, \lambda_1\}$. Thus, it follows that, for all $z \in K$, the relation 
\[
\omega_{\alpha}(z+t,K,D,\psi)-C_2 \delta^{\lambda_2} \leq \psi(z)
\]
holds. As the function $\rho(z)$ satisfies the Lipwitz condition in $G$, and from the conditions $|t| \leq \delta$ and $\rho|_{\partial D}=0$, we find that, for the points $z$ with $z+t \not \in D$, the relation
\[\tilde \omega(z+t)=\rho(z+t) \leq L \delta\]
holds. Therefore, the function $\tilde \omega(z+t)-C_3 \delta^{\lambda_2}$ belongs to the class $\mathcal{U}(K,D,\psi)$, where $C_3 =\max\{L,C_2\}$. As a result, we get
\[
\tilde \omega(z+t)-C_3 \delta^{\lambda_2} \leq \omega_{\alpha}(z,K,D,\psi), \quad z \in D.
\]
Apart from that, for all $z \in D$ with $z+t \in D$, we have 
\begin{equation}\label{inequality1}
\omega_{\alpha}(z+t,K,D,\psi)-\omega_{\alpha}(z,K,D,\psi) \leq C_3 \delta^{\lambda_2}.   
\end{equation}
Note that, the point $w\in K$ satisfies the condition $dist(z,K)=|z-w|$. Therefore, it is evident that
\[|z-w| \leq \delta\] 
for $z +t \in K$. In addition, by the triangle inequality, we get that
\[|z+t-w| \leq 2 \delta.\]
Thus, for $z \in D, \, z+t \in K$, we have following estimate
\[
|\omega_{\alpha}(z,K,D, \psi)-\psi(z+t)| \leq |\omega_{\alpha}(z,K,D,\psi)-\psi(w)|+|\psi(w)-\psi(z+t)| \leq C_2 \delta^{\lambda_2}.
\]
As a result, we get
\[\omega_{\alpha}(z,K,D,\psi)-C_2 \delta^{\lambda_2} \leq \psi(z+t), \quad z+t \in K.\] 
Consequently, we obtain that 
\[
\tilde \omega(z)-C_3 \delta^{\lambda_2} \leq \omega_{\alpha}(z+t,K,D,\psi), \quad z+t \in D.
\]
Therefore, for $z \in D$, $z+t \in D$, we have that
\begin{equation}\label{inequality2}
\omega_{\alpha}(z,K,D,\psi)-\omega_{\alpha}(z+t,K,D,\psi) \leq C_3 \delta^{\lambda_2}.   
\end{equation}
Hence, from the relations \eqref{inequality1} and \eqref{inequality2}, the H\"older continuity of the function $\omega_{\alpha}(z,K,D,\psi)$ in $D$ follows.
The proof is complete.
\end{proof}

The following result follows from the proof of Theorem~\ref{Holder}.

\begin{corollary}
Let \(D \subset \mathbb{C}^n\) be a strongly \(\alpha\)-regular domain, and let
\(K\subset D\) be a compact set. If \(\omega_\alpha^*(z,K,D)\) is H\"older
continuous in \(D\), then it admits an \(\alpha\)-subharmonic and H\"older
continuous extension to some neighborhood of \(\overline{D}\).
\end{corollary}

Moreover, as an immediate consequence of Theorem~\ref{Holder}, by taking
$\psi\equiv -1$, we obtain the following corollary, which corresponds to the
unweighted case.

\begin{corollary}
Let \(D \subset \mathbb{C}^n\) be a strongly \(\alpha\)-regular domain, and let
\(K\subset D\) be a compact set. Then the function
\(\omega_\alpha(z,K,D)\) is H\"older continuous in \(D\) if and only if there
exist constants \(C>0\) and \(\gamma\in(0,1]\) such that
$$1+\omega_{\alpha}(z,K,D)
\leq C\,\operatorname{dist}(z,K)^\gamma$$
holds in some neighborhood of \(K\).
\end{corollary}


\begin{thebibliography}{99}  
\bibitem{ABIS} Abdullaev B., Imomkulov S., Sharipov R., \emph{$\alpha-$subharmonic functions.} Contemporary Mathematics. Fundamental directions. 67 (2021), no.4, p. 620--633. 
\bibitem{ABSA} Abdullaev B., Sadullaev A., \emph{Potential theory in the class of $m \text{-}sh$ functions.} Proceedings of the Steklov Institute of Mathematics. 279 (2012), no. 1,  p. 155--180.
\bibitem{ABSR} Abdullaev B., Sharipov R., \emph{Local and global $\alpha-$polar sets.} Bulletin of the Institute of Mathematics. 5 (2019),  p. 4--8.
\bibitem{BZ} Blocki Z., \emph{Weak solutions to the complex Hessian equation.} Ann. Ins. Fourier, Grenoble. 55 (2005), no.5, p. 1735--1756.

\bibitem{DSKS1} Dinew S. and  Kolodziej S.,  \emph{ A priori estimates for the complex Hessian equation.} Analysis and PDE. 7 (2014), no. 1,  p. 227--244.

\bibitem{DSKS2}
Dinew S.  and  Kolodziej S.,   \emph{Non standard properties of $m$-subharmonic functions.} Dolomites Research Notes on Approximation. 11(2018), no. 4,  p. 35--50.


\bibitem{JP} Demailly J.-P. , \emph{Complex analytic and differential geometry.} Universite de Grenoble, France, 2012.
\bibitem{DT} Gilbarg D., Trudinger N., \emph{Elliptic partial differential equations of second order.} New York, Springer, 1998.


\bibitem{KK} Kolmogorov A. N., Fomin S. V., \emph{Elements of the theory of functions and functional analysis.} Graylock Press, Albany N.Y., Vol. 2.  1961.  

\bibitem{KKRK} Kuldoshev K. and Rakhimov K., \emph{On the H\"older continuity of weighted extremal functions.} Journal of Siberian Federal University. Mathematics and Physics. 18 (2025),  no. 6, p. 809--818.
\bibitem{MR} Miranda C., \emph{Partial differential equations of elliptic type.} Moscow, 1957. 

\bibitem{SAZA} Sadullaev A.,  Zeriahi A.,  \emph{Hölder regularity of geneic manifolds.} Ann. Sc. Norm. Super. Pisa Cl. Sci. 16 (2016), no. 5, p. 369--382.

\bibitem{AS} Sadullaev A., \emph{Pluripotential  theory. Applications}. Palmarium Academic Publishing, 2012.

\bibitem{SAA} Sadullaev A., \emph{Plurisubharmonic measures and capacities on complex manifolds.} Russian Math. Surveys.  36 (1981), no. 4, p. 61--119.
\bibitem{SR} Sharipov R., \emph{$P-$measure and $P-$capacity in the class of $\alpha-$subharmonic functions.} Lectures of the Academy of Sciences of the Republic of Uzbekistan. 3 (2019),   p. 11--15. 

\bibitem{JS1} Siciak J., \emph {Wiener’s type sufficient conditions in $\mathbb{C}^n$. }  Universitatis Iagellonicae Acta Mathematica. 35 (1997),  p. 151-161.
\bibitem{VM} Vaisova M., \emph{ Potential theory in the class of $\alpha-$subharmonic functions.} Uzbek Mathematical Journal, 3(2016), p. 46--52.
\bibitem{WH} Whitney H., \emph{Analytic extensions of differentiable functions defined in closed sets.} Transactions of the American Mathematical Society 36 (1934), p. 63--89.

\end{thebibliography}
\end{document}